\documentclass[12pt,reqno]{amsart}
\usepackage{etex,mathtools}
\usepackage{mathdots,longtable}
\usepackage{amscd,amssymb,amsmath,multicol,tikz,float,mathrsfs}
\usepackage[left=1in,right=1in]{geometry}
\begin{document}

\restylefloat{table}
\newtheorem{thm}[equation]{Theorem}
\numberwithin{equation}{section}
\newtheorem{cor}[equation]{Corollary}
\newtheorem{expl}[equation]{Example}
\newtheorem{rmk}[equation]{Remark}
\newtheorem{conv}[equation]{Convention}
\newtheorem{claim}[equation]{Claim}
\newtheorem{lem}[equation]{Lemma}
\newtheorem{sublem}[equation]{Sublemma}
\newtheorem{conj}[equation]{Conjecture}
\newtheorem{defin}[equation]{Definition}
\newtheorem{diag}[equation]{Diagram}
\newtheorem{prop}[equation]{Proposition}
\newtheorem{notation}[equation]{Notation}
\newtheorem{tab}[equation]{Table}
\newtheorem{fig}[equation]{Figure}
\newcounter{bean}
\renewcommand{\theequation}{\thesection.\arabic{equation}}

\raggedbottom \voffset=-.7truein \hoffset=0truein \vsize=8truein
\hsize=6truein \textheight=8truein \textwidth=6truein
\baselineskip=18truept
\def\mapleft#1{\smash{\mathop{\longleftarrow}\limits^{#1}}}
\def\mapright#1{\ \smash{\mathop{\longrightarrow}\limits^{#1}}\ }
\def\ml#1{\,\smash{\mathop{\leftarrow}\limits^{#1}}\,}
\def\mapup#1{\Big\uparrow\rlap{$\vcenter {\hbox {$#1$}}$}}
\def\mapdown#1{\Big\downarrow\rlap{$\vcenter {\hbox {$\ssize{#1}$}}$}}
\def\mapne#1{\nearrow\rlap{$\vcenter {\hbox {$#1$}}$}}
\def\mapse#1{\searrow\rlap{$\vcenter {\hbox {$\ssize{#1}$}}$}}
\def\mapr#1{\smash{\mathop{\rightarrow}\limits^{#1}}}
\def\Mt{\widetilde{\M}}
\def\ss{\smallskip}
\def\s{\sigma}
\def\Bt{\widetilde{\mathcal{B}}}
\def\l{\lambda}
\def\Ah{\widehat{A}}
\def\Bh{\widehat{B}}
\def\vp{v_1^{-1}\pi}
\def\at{{\widetilde\alpha}}
\def\At{\widetilde{\mathcal{A}}}
\def\as{\mathscr{A}}
\def\Ast{\widetilde{\as}}
\def\Mct{\widetilde{\mathcal{M}}}
\def\sm{\wedge}
\def\la{\langle}
\def\ra{\rangle}
\def\lar{\leftarrow}
\def\ev{\text{ev}}
\def\od{\text{od}}
\def\on{\operatorname}
\def\ol#1{\overline{#1}{}}
\def\spin{\on{Spin}}
\def\cat{\on{cat}}
\def\Lbar{\overline{\Lambda}}
\def\qed{\quad\rule{8pt}{8pt}\bigskip}
\def\ssize{\scriptstyle}
\def\a{\alpha}
\def\bz{{\Bbb Z}}
\def\Rhat{\hat{R}}
\def\im{\on{im}}
\def\ct{\widetilde{C}}
\def\ext{\on{Ext}}
\def\sq{\on{Sq}}
\def\eps{\epsilon}
\def\ar#1{\stackrel {#1}{\rightarrow}}
\def\br{{\bold R}}
\def\bC{{\bold C}}
\def\bA{{\bold A}}
\def\bB{{\bold B}}
\def\bD{{\bold D}}
\def\bC{{\bold C}}
\def\bh{{\bold H}}
\def\bQ{{\bold Q}}
\def\bP{{\bold P}}
\def\bx{{\bold x}}
\def\bo{{\bold{bo}}}
\def\dh{\widehat{d}}
\def\A{\mathcal{A}}
\def\B{\mathcal{B}}
\def\si{\sigma}
\def\Vbar{{\overline V}}
\def\dbar{{\overline d}}
\def\wbar{{\overline w}}
\def\Sum{\sum}
\def\tfrac{\textstyle\frac}

\def\tb{\textstyle\binom}
\def\Si{\Sigma}
\def\w{\wedge}
\def\equ{\begin{equation}}
\def\b{\beta}
\def\G{\Gamma}
\def\L{\Lambda}
\def\g{\gamma}
\def\d{\delta}
\def\k{\kappa}
\def\psit{\widetilde{\Psi}}
\def\tht{\widetilde{\Theta}}
\def\psiu{{\underline{\Psi}}}
\def\thu{{\underline{\Theta}}}
\def\aee{A_{\text{ee}}}
\def\aeo{A_{\text{eo}}}
\def\aoo{A_{\text{oo}}}
\def\aoe{A_{\text{oe}}}
\def\vbar{{\overline v}}
\def\endeq{\end{equation}}
\def\xhat{\widehat{x}}
\def\sn{S^{2n+1}}
\def\zp{\bold Z_p}
\def\cR{{\mathcal R}}
\def\P{{\mathcal P}}
\def\cQ{{\mathcal Q}}
\def\cj{{\cal J}}
\def\zt{{\bold Z}_2}
\def\bs{{\bold s}}
\def\bof{{\bold f}}
\def\bq{{\bold Q}}
\def\be{{\bold e}}
\def\Hom{\on{Hom}}
\def\ker{\on{ker}}
\def\kot{\widetilde{KO}}
\def\coker{\on{coker}}
\def\da{\downarrow}
\def\colim{\operatornamewithlimits{colim}}
\def\zphat{\bz_2^\wedge}
\def\io{\iota}
\def\om{\omega}
\def\Prod{\prod}
\def\e{{\cal E}}
\def\zlt{\Z_{(2)}}
\def\exp{\on{exp}}
\def\abar{{\overline a}}
\def\xbar{{\overline x}}
\def\ybar{{\overline y}}
\def\zbar{{\overline z}}
\def\mbar{{\overline m}}
\def\nbar{{\overline n}}
\def\sbar{{\overline s}}
\def\kbar{{\overline k}}
\def\bbar{{\overline b}}
\def\et{{\widetilde E}}
\def\ni{\noindent}
\def\tsum{\textstyle \sum}
\def\coef{\on{coef}}
\def\den{\on{den}}
\def\lcm{\on{l.c.m.}}
\def\Ext{\operatorname{Ext}}
\def\iso{\approx}
\def\lra{\longrightarrow}
\def\vi{v_1^{-1}}
\def\ot{\otimes}
\def\psibar{{\overline\psi}}
\def\thbar{{\overline\theta}}
\def\Mh{{\widehat M}}
\def\exc{\on{exc}}
\def\ms{\medskip}
\def\ehat{{\hat e}}
\def\etao{{\eta_{\text{od}}}}
\def\etae{{\eta_{\text{ev}}}}
\def\dirlim{\operatornamewithlimits{dirlim}}
\def\gt{\widetilde{L}}
\def\lt{\widetilde{\lambda}}
\def\st{\widetilde{s}}
\def\ft{\widetilde{f}}
\def\sgd{\on{sgd}}
\def\lfl{\lfloor}
\def\rfl{\rfloor}
\def\ord{\on{ord}}
\def\gd{{\on{gd}}}
\def\rk{{{\on{rk}}_2}}
\def\nbar{{\overline{n}}}
\def\MC{\on{MC}}
\def\lg{{\on{lg}}}
\def\cH{\mathcal{H}}
\def\cS{\mathcal{S}}
\def\cP{\mathcal{P}}
\def\N{{\Bbb N}}
\def\Z{{\Bbb Z}}
\def\Q{{\Bbb Q}}
\def\R{{\Bbb R}}
\def\C{{\Bbb C}}
\def\Lb{\overline\Lambda}
\def\mo{\on{mod}}
\def\xt{\times}
\def\notimm{\not\subseteq}
\def\Remark{\noindent{\it  Remark}}
\def\kut{\widetilde{KU}}
\def\Eb{\overline E}
\def\*#1{\mathbf{#1}}
\def\0{$\*0$}
\def\1{$\*1$}
\def\22{$(\*2,\*2)$}
\def\33{$(\*3,\*3)$}
\def\ss{\smallskip}
\def\ssum{\sum\limits}
\def\dsum{\displaystyle\sum}
\def\la{\langle}
\def\ra{\rangle}
\def\on{\operatorname}
\def\proj{\on{proj}}
\def\od{\text{od}}
\def\ev{\text{ev}}
\def\o{\on{o}}
\def\U{\on{U}}
\def\lg{\on{lg}}
\def\a{\alpha}
\def\bz{{\Bbb Z}}
\def\ccM{{\Bbb M}}
\def\E{\mathcal{E}}
\def\eps{\varepsilon}
\def\bc{{\bold C}}
\def\bN{{\bold N}}
\def\bB{{\bold B}}
\def\bW{{\bold W}}
\def\nut{\widetilde{\nu}}
\def\tfrac{\textstyle\frac}
\def\b{\beta}
\def\G{\Gamma}
\def\g{\gamma}
\def\zt{{\Bbb Z}_2}
\def\zth{{\bold Z}_2^\wedge}
\def\bs{{\bold s}}
\def\bx{{\bold x}}
\def\bof{{\bold f}}
\def\bq{{\bold Q}}
\def\be{{\bold e}}
\def\lline{\rule{.6in}{.6pt}}
\def\xb{{\overline x}}
\def\xbar{{\overline x}}
\def\ybar{{\overline y}}
\def\zbar{{\overline z}}
\def\ebar{{\overline e}}
\def\nbar{{\overline n}}
\def\ubar{{\overline u}}
\def\bbar{{\overline b}}
\def\et{{\widetilde e}}
\def\M{\mathcal{M}}
\def\lf{\lfloor}
\def\rf{\rfloor}
\def\ni{\noindent}
\def\ms{\medskip}
\def\Dhat{{\widehat D}}
\def\what{{\widehat w}}
\def\Yhat{{\widehat Y}}
\def\abar{{\overline{a}}}
\def\minp{\min\nolimits'}
\def\sb{{$\ssize\bullet$}}
\def\mul{\on{mul}}
\def\N{{\Bbb N}}
\def\Z{{\Bbb Z}}
\def\Q{{\Bbb Q}}
\def\R{{\Bbb R}}
\def\C{{\Bbb C}}
\def\Xb{\overline{X}}
\def\eb{\overline{e}}
\def\notint{\cancel\cap}
\def\cS{\mathcal S}
\def\cR{\mathcal R}
\def\el{\ell}
\def\TC{\on{TC}}
\def\GC{\on{GC}}
\def\wgt{\on{wgt}}
\def\Ht{\widetilde{H}}
\def\wbar{\overline w}
\def\dstyle{\displaystyle}
\def\Sq{\on{sq}}
\def\Om{\Omega}
\def\ds{\dstyle}
\def\tz{tikzpicture}
\def\zcl{\on{zcl}}
\def\bd{\bold{d}}
\def\cM{\mathcal{M}}
\def\io{\iota}
\def\Vb#1{{\overline{V_{#1}}}}
\def\Ebar{\overline{E}}
\def\lb{\,\begin{picture}(-1,1)(1,-1)\circle*{3.5}\end{picture}\ }
\def\rlb{\,\begin{picture}(-1,1)(1,-1) \circle*{4.5}\end{picture}\ }
\def\lbb{\,\begin{picture}(-1,1)(1,-1)\circle*{8}\end{picture}\ }
\def\zp{\Z_p}
\def\lbr{\,\begin{picture}(-1,1)(1,-1)[dashed]\circle*{3.5}\end{picture}\ }
\def\llb{\,\begin{picture}(-1,1)(1,-1)\circle*{2.6}\end{picture}\ }
\def\blb{\,\begin{picture}(-1,1)(1,-1) \circle*{5.8}\end{picture}\ }
\def\alh{\widehat{\a}}
\setcounter{MaxMatrixCols}{15}
\title
{Orientable manifolds with nonzero dual Stiefel-Whitney classes of largest possible grading}
\author{Donald M. Davis}
\address{Department of Mathematics, Lehigh University\\Bethlehem, PA 18015, USA}
\email{dmd1@lehigh.edu}
\date{July 29, 2025}
\begin{abstract} It is known that, for all $n$, there exist compact differentiable orientable $n$-manifolds with dual Stiefel-Whitney class $\wbar_{n-\alh(n)}\ne0$, and this is best possible, but the proof is nonconstructive. Here $\alh(n)$ equals the number of 1's in the binary expansion of $n$ if $n\equiv1$ mod 4, and exceeds this by 1 otherwise. We find, for all $n\not\equiv0$  mod 4, examples of real Bott maanifolds with this property.\end{abstract}
\maketitle
\section{Introduction}\label{intro}
The dual Stiefel-Whitney classes $\wbar_i(M)$ of a compact differentiable\footnote{All manifolds considered in this paper will be compact and differentiable, and we will henceforth use the word ``manifold'' to mean ``compact differentiable manifold.'' All cohomology groups have coefficients in $\zt=\Z/2$.} $n$-manifold $M$ are elements of $H^i(M;\zt)$ which have the property that if $\wbar_i(M)\ne0$, then $M$ cannot be embedded in $\R^{n+i}$ nor immersed in $\R^{n+i-1}$.
An important early result in algebraic topology was the following theorem of Massey.
\begin{thm} \label{Masthm} $($\cite{Mas}$)$ Let $\a(n)$ denote the number of $1$'s in the binary expansion of $n$. If $M$ is an  $n$-manifold and $i>n-\a(n)$, then $\wbar_i(M)=0$.
Moreover, there exists an $n$-manifold $M$ with $\wbar_{n-\a(n)}(M)\ne0$.
\end{thm}
This theorem was the first step toward Cohen's theorem (\cite{C}) that every $n$-manifold can be immersed in $\R^{2n-\a(n)}$.
One attractive feature of this theorem is that the existence part is realized by a well-known manifold: if $n=2^{e_1}+\cdots+2^{e_{\a(n)}}$ with $e_i$ distinct,  then $M=\prod RP^{2^{e_i}}$ is an $n$-manifold with $\wbar_{n-\a(n)}(M)\ne0$.

In \cite{DW}, the following analogue for orientable manifolds was proved.
\begin{thm}\label{DWthm} $($\cite[Theorem 1.1]{DW}$)$ Let
$$\alh(n)=\begin{cases}\a(n)&n\equiv1\ (4)\\
\a(n)+1&n\not\equiv1\ (4).\end{cases}.$$
If $M$ is an orientable $n$-manifold and $i>n-\alh(n)$, then $\wbar_i(M)=0$.
Moreover, there exists an orientable $n$-manifold $M$ with $\wbar_{n-\alh(n)}(M)\ne0$.\end{thm}
This theorem was a reformulation of a result of \cite{Papa}, which drew heavily from results in \cite{MP} and \cite{BP}. The existence part was nonconstructive, and the goal of this paper is to find explicit orientable $n$-manifolds with $\wbar_{n-\alh(n)}(M)\ne0$. Our main theorem, \ref{main}, accomplishes this with real Bott manifolds for all $n\not\equiv0$ mod 4.  
 If $n$ is a 2-power $\ge4$, complex projective space $CP^{n/2}$ provides an example of an orientable $n$-manifold with $\wbar_{n-\alh(n)}\ne0$. For $n\equiv0$ mod 4 and not a 2-power, we do not know explicit orientable $n$-manifolds with $\wbar_{n-\alh(n)}\ne0$.  See Section \ref{mfsec}.

Real Bott manifolds have been studied recently by many authors, for example (\cite{Choi}, \cite{Ds}). The following known result describes the properties relevant to our work.

\begin{thm} If $A=(a_{i,j})$ is an $n$-by-$n$ strictly upper-triangular binary matrix, there is an $n$-manifold $B_{n}$  with
$$H^*(B_{n};\zt)\approx\zt[x_1,\ldots,x_n]/(x_j^2=\sum_{i=1}^{j-1}a_{i,j}x_ix_j),$$ where $|x_i|=1$.
The manifold is orientable if and only if each row of $A$ contains an even number of $1$'s.
\end{thm}
\begin{proof} The manifold is formed from a Bott tower
$$B_n\to B_{n-1}\to\cdots\to B_1\to B_0=\{*\}.$$
Each $B_j$ is the projectification of $\xi_j\oplus\eps$, where $\xi_j$ and $\eps$ are line bundles over $B_{j-1}$, with $\eps$ trivial, and $w_1(\xi_j)=\dsum_{i=1}^{j-1}a_{i,j}x_i$. The cohomology result is \cite[Lemma 2.1]{KM}. The orientability claim is well-known. For example, \cite[Proposition 2.5]{Ds}.\end{proof}

Our main theorem, \ref{main}, gives the desired examples  when $n\equiv1$ mod 4, while  Corollary \ref{cormain} gives the examples when $n\equiv2,3$ mod 4.
\begin{thm} Let $n\equiv1\mod 4$. The Bott manifold $B_n$ with relations
$x_n^2=(x_1+\cdots+x_{n-2})x_n$, $x_i^2=x_{i-1}x_i$ for $2\le i\le n-1$, and $x_1^2=0$ is orientable with $\wbar_{n-\alh(n)}\ne0$.\label{main}
\end{thm}

For example, the Bott manifold $B_5$ corresponds to the following matrix:
$$\begin{pmatrix}
0&1&0&0&1\\
0&0&1&0&1\\
0&0&0&1&1\\
0&0&0&0&0\\
0&0&0&0&0\end{pmatrix}$$

\begin{cor} For $n\equiv2$ $($resp.~$3)$ $\mod 4$, $B_{n-1}\times S^1$ (resp.~$B_{n-2}\times S^1\times S^1$) is an orientable  $n$-dimensional Bott manifold with $\wbar_{n-\alh(n)}\ne0$. Here $B_{n-1}$ and $B_{n-2}$ are the manifolds of Theorem \ref{main}.\label{cormain}\end{cor}
\begin{proof} This is immediate from the Whitney product formula, triviality of the tangent bundle of $S^1$, Theorem \ref{main}, and that if $m\equiv1$ mod 4, then $m-\alh(m)=m+1-\alh(m+1)=m+2-\alh(m+2)$. These are Bott manifolds by appending one or two steps to the Bott tower with $\xi_j$ trivial.
\end{proof}

\section{Proof of Theorem \ref{main}}\label{pfsec}
The case $n=1$ has $B_n\approx S^1$ with $n-\alh(n)=0$, and $\wbar_0(S^1)\ne0$. From now on, we assume $n>1$. Recall that $\alh(n)=\a(n)$ since $n\equiv1$ mod 4.

We will prove that there is a class $z\in H^{\a(n)}(B_n)$ such that $\chi\sq^{n-\a(n)}(z)\ne0$. Here $\chi$ is the antiautomorphism of the mod 2 Steenrod algebra.
This implies Theorem \ref{main} by the following standard argument (e.g., \cite[\S3]{MP}). The $S$-dual of the manifold $B_n$ is the Thom space of its stable normal bundle, and dual to $\chi\sq^{n-\a(n)}$ into the top class of the manifold is $\sq^{n-\a(n)}$ on the Thom class  of the normal bundle. But this equals $w_{n-\a(n)}$ of the normal bundle, which  is $\wbar_{n-\a(n)}$ of the manifold.

By \cite{Mil}, $\chi\sq^{n-\a(n)}$ equals the sum of all Milnor basis elements of grading $n-\a(n)$.
Let $n=2^{p_r}+\cdots+2^{p_1}+1$ with $p_r>\cdots> p_1\ge2$ and $r=\a(n)-1$. Then $n-\a(n)=2^{p_r}+\cdots+2^{p_1}-r$. A Milnor basis element $\sq(t_1,\ldots,t_s)$ has grading $\sum t_i(2^i-1)$ and vanishes on classes of grading less than $\sum t_i$. We omit the proof of the following lemma, which is very similar to the second half of the proof of \cite[Theorem 1.3(c)]{DW}.

\begin{lem} A tuple $(t_1,\ldots,t_s)$ satisfies $\sum t_i(2^i-1)=2^{p_r}+\cdots +2^{p_1}-r$ and $\sum t_i\le r$ if and only if
$$t_i=\begin{cases}1&i=p_j \text{ for some }j\\ 0&\text{otherwise.}\end{cases}$$\label{Blem}
\end{lem}

Let $\b$ denote the Milnor basis element $\sq(t_1,\ldots,t_{p_r})$ for the $t_i$'s of Lemma \ref{Blem}.
Then Lemma \ref{Blem} and the preceding paragraph imply that $\chi\sq^{n-\a(n)}(z)=\b(z)$ on any class $z$ with $|z|\le r$.

For typographical reasons, we let $P_i=2^{p_i}$, so $n=P_r+\cdots+P_1+P_0$ with $P_0=1$.
A basis for $H^{\a(n)}(B_n)$ consists of all $x_{i_1}\cdots x_{i_{\a(n)}}$ with $1\le i_1<\cdots<i_{\a(n)}\le n$. Our class $z$ is chosen to be the one with 
\begin{equation}\label{idef}i_j=\begin{cases}P_j+\cdots+P_1&1\le j\le r\\
n&j=\a(n)=r+1.\end{cases}\end{equation}

Recall that $\cP_t^0$ denotes the Milnor basis element with only nonzero entry a 1 in position $t$, and that $\cP_t^0(x)=x^{2^t}$ if $|x|=1$. Using this and the Cartan formula, $\chi\sq^{n-\a(n)}(z)$
is the sum  over all permutations of $\{P_0,\ldots,P_r\}$ of products of $x_{i_1},\ldots,x_{i_{\a(n)}}$ raised to these 2-power exponents. Here, $i_j$ are as defined in (\ref{idef}).  See (\ref{permsum}) for a restatement involving notation defined in the interim.

The relations in $H^*(B_n)$ on classes $x_i$ for $i<n$ imply that
\begin{equation}\label{xie}x_i^e=\begin{cases}x_{i-e+1}\cdots x_i&e\le i\\ 0&e>i.\end{cases}\end{equation}
Also, $x_n^{2^p}=(x_1+\cdots+x_{n-2})^{2^p-1}x_n$, as is easily proved by induction on  $p$.
We will use these facts repeatedly. 

As a warmup, we consider the case $\a(n)=2$, so $n=P+1$ with $P=2^p$, $P\ge4$. Then $z=x_Px_n$, and
$$\chi\sq^{n-\a(n)}(z)=x_P^Px_n+x_Px_n^P=x_1\cdots x_n+(x_1+\cdots+x_{n-2})^{P-1}x_{n-1}x_n.$$
Since $x_1\cdots x_n$ is the nonzero element of $H^n(B_n)$, the case $\a(n)=2$ of Theorem \ref{main} follows from the following key result.

\begin{thm}\label{key} Let $m=2^p-1$ with $p>1$. Let $Q_m$ denote the ring $\zt[x_1,\ldots,x_m]$ with relations $x_i^2=x_{i-1}x_i$ for $2\le i\le m$, and $x_1^2=0$. Then
$$(x_1+\cdots+x_m)^m=0.$$
\end{thm}

We precede the proof of Theorem \ref{key} with two useful results.
\begin{lem} Let $Q_m$ be as in Theorem \ref{key}.
\begin{itemize}
\item A degree-$m$ monomial $x_1^{e_1}\cdots x_m^{e_m}$  equals $0$
 in $Q_m$  iff for some $d$, $e_1+\cdots+e_d>d$.
\item If $d$ is maximal with respect to this property, then $e_{d+1}=0$. This is called a {\em gap} at $x^{d+1}$.
\end{itemize}
\end{lem}
\begin{proof} We define excess numbers $\Delta_i=e_1+\cdots+e_i-i$. The excess sequence  $(\Delta_1,\ldots,\Delta_{m-1},\Delta_m=0)$ satisfies $\Delta_i\ge \Delta_{i-1}-1$, as their difference is $e_i$. The effect of a relation is to increase $\Delta_{i-1}$ by 1 if $\Delta_i>\Delta_{i-1}$. The only way to  create a positive $\Delta$ is if $\Delta_i=0$ and $\Delta_{i+1}>0$. Thus if the excess sequence starts with no positive $\Delta$'s, it can never have any. But such a sequence can be reduced to the 0-sequence since the negative entry with largest subscript can always be increased by 1. This shows that if a monomial has no $d$ as in the lemma, then the monomial equals $x_1\cdots x_m\ne0$ in $Q_m$.

On the other hand, if there is a $d$ as in the lemma, then the excess sequence has a positive entry.  Let $\Delta_i$ be the first positive entry. Then we can increase $\Delta_{i-1}$ until it equals $\Delta_i$. We can continue this until making $\Delta_1>0$, which says the monomial equals 0.

The second part of the lemma is clear, since having $e_{d+1}>0$ would contradict maximality of $d$.
\end{proof}
\begin{cor} A degree-$m$ monomial equals $0$ in $Q_m$ iff it can be decomposed as $A\cdot B$, where $A$ has degree $D$ but involves just $x_i$ for $i<D$, while $B$ is equivalent to $x_{D+1}\cdots x_m$ in $Q_m$.\label{cor}
\end{cor}
\begin{proof}
Here $D$ is the $e_1+\cdots+ e_d$ of the lemma for maximal $d$, and $A=x_1^{e_1}\cdots x_d^{e_d}$. Then $B=x_{D+1}^{e_{D+1}}\cdots x_m^{e_m}$ has degree $m-D$. Maximality of $d$ implies that the excess  numbers of $B$ are $\le -D$, and its last excess number equals $-D$. Consideration of the last excess number less than $-D$ shows that the excess sequence of $B$ can be reduced to $(-D,\ldots,-D)$, implying the claim about $B$.
\end{proof}

 We illustrate with the monomial $x_2^3x_3x_6^2$ in $Q_6$. The maximal $d$ of the lemma is 3, while $D$ of the corollary is $4$. We have $A=x_2^3x_3$, while $B=x_5^0x_6^2$ with initial excess sequence $(-5,-4)$.

\begin{proof}[Proof of Theorem \ref{key}] The polynomial $(x_1+\cdots+x_m)^m$ can be expanded as
\begin{equation}\label{expand}\prod_{i=0}^{p-1}(x_{2^i}^{2^i}+\cdots+x_m^{2^i}),\end{equation} using (\ref{xie}).
The total number of monomials when this is expanded is $\prod_{i=0}^{p-1}(m+1-2^i)$, which is even since $p>1$. 
We will enumerate the number of these monomials which equal 0 in $Q_m$.

Those having a gap at $x^D$ with $D=E_1+\cdots E_s$ for distinct 2-powers $E_i$ have their $A$ of Corollary \ref{cor} equal to any monomial in the expansion of
$$\prod_{i=1}^s(x_{E_i}^{E_i}+\cdots+x_{D-1}^{E_i}).$$
The number of such monomials is even, using the factor for $E_i=1$ if $D$ is odd, and any factor if $D$ is even. For any fixed $D$, each monomial $A$ will occur multiplied by the same set of monomials $B$. Thus the number of monomials in the original expansion having a gap at $x^D$ is even, and so the number of monomials with some gap is even, and this is the number which equal 0 in $Q_m$. Since the total number of monomials was even, we deduce that the number which are nonzero is even. Since every nonzero monomial equals $x_1\cdots x_m$ in $Q_m$, we deduce that the polynomial equals 0.

\end{proof}

\begin{defin} 
We introduce some notation in $H^*(B_n)$.
\begin{itemize}
\item $T_i=P_i+\cdots+P_1$ for $i\ge 1$ and $T_0=0$. Here $P_j=2^{p_j}$ as in (\ref{idef}).
\item $S=x_1+\cdots+x_{n-2}$;
\item $[i,j]=x_i\cdots x_j$ and $(i,j]=x_{i+1}\cdots x_j$;
\item If $W$ is a homogeneous polynomial of degree $d$, then $L_t$ (in $L_tW$) denotes any product of $n-d$ classes $x_i$ (not necessarily distinct) with $i\le t$.
\end{itemize}
\end{defin}

Also, recall that $n=P_r+\cdots+P_1+1$.
\begin{thm}\label{zero} The following elements of $H^n(B_n)$ equal $0$.
\begin{itemize}
\item[a.] $L_{T_j-P_i-1}\cdot(T_j-P_i,n]\cdot S^{P_j-1}$, $1\le i<j$;
\item[b.] $[1,T_{j-1}]\cdot [T_j,n]\cdot S^{P_j-1}$, $j\ge1$.
\end{itemize}    
\end{thm}

\begin{proof}

[a.] Let $D_j=T_j-P_i$. Write
\begin{equation} \label{SP-1}S^{P_j-1}=\prod_{k=0}^{p_j-1}(x_1^{2^k}+\cdots+x_{n-2}^{2^k}).\end{equation}
For any $k$, and $\ell$ satisfying $D_j\le\ell\le n-2$, 
$$(x_{D_j+1}\cdots x_n)\cdot x_\ell^{2^k}=x_{D_j+1-2^k}\cdots x_n,$$
independent of $\ell$. Since the number of values of $\ell$ is even,
$$(x_{D_j+1}\cdots x_n)\cdot(x_{D_j}^{2^k}+\cdots+x_{n-2}^{2^k})=0.$$
We are left with
$$L_{D_j-1}\cdot (x_{D_j+1}\cdots x_n)\cdot\prod _{k=0}^{p_j-1}(x_1^{2^k}+\cdots+x_{D_j-1}^{2^k}).$$
 This is 0 because it has a gap at $x_{D_j}$. The part complementary to the middle factor has degree $D_j$, which is greater than its largest subscript.

[b.] The argument begins by reducing $S$  to $(x_1+\cdots+x_{T_j-1})$, similarly to part (a). [\![Write $S^{P_j-1}$ as in (\ref{SP-1}). The even number of terms $x_\ell^{2^k}$ in each factor with $T_j\le\ell\le n-2$ can be ignored in the presence of $x_{T_j}\cdots x_n$.]\!]
Now we have
$$(x_1\cdots x_{T_{j-1}})\cdot(x_{T_j}\cdots x_n)\cdot(x_1+\cdots+x_{T_j-1})^{P_j-1}.$$
Monomials in $(x_1+\cdots+x_{T_j-1})^{P_j-1}$ containing any factor $x_i$ with $i\le T_{j-1}$ yield 0 upon multiplication by $x_1\cdots x_{T_{j-1}}$. So now we have
$$(x_1\cdots x_{T_{j-1}})\cdot(x_{T_j}\cdots x_n)\cdot(x_{T_{j-1}+1}+\cdots+x_{T_j-1})^{P_j-1}.$$
Because $x_{T_{j-1}+1}^2=0$ in the presence of $x_1\cdots x_{T_{j-1}}$, we see that $x_{T_{j-1}+1}$ is acting here like $x_1$ does in $Q_m$. The uniformity of the other relations implies that we can effectively reduce the subscripts in $(x_{T_{j-1}+1}+\cdots +x_{P_j+T_{j-1}-1})^{P_j-1}$ by $T_{j-1}$, so that it behaves like $(x_1+\cdots+x_{P_j-1})^{P_j-1}$, which equals 0 by Theorem \ref{key}.

\end{proof}

Now we prove Theorem \ref{main}. There are $2^r$ bijections
$$\{1,\ldots,r+1\}\mapright{\s}\{0,\ldots,r\}$$
satisfying $\s(i)\le i$ for all $i$. [\![$\s(1)\in\{0,1\}$, $\s(i)\in \{0,\ldots,i\}-\{\s(1),\ldots,\s(i-1)\}$, and then $\s(r+1)$ is uniquely determined.]\!]
Then 
\begin{equation}\label{permsum}\chi\sq^{n-\a(n)}(x_{T_1}\cdots x_{T_r}x_n)=\sum_{\s}\biggl(\bigl(\prod_{i=1}^rx_{T_i}^{P_{\s(i)}}\bigr) x_n^{P_{\s(r+1)}}\biggr).\end{equation}

We divide the $2^r$ terms of (\ref{permsum}) into four types, determined by $\s^{-1}(r-1)$ and $\s^{-1}(r)$. The first type is when $\s(r)=r-1$ and $\s(r+1)=r$. Then this summand of (\ref{permsum}) equals $\cdots x_{T_{r-1}}^{P_i}x_{T_r}^{P_{r-1}}x_n^{P_r}$, with $i<r-1$. In the notation of Theorem \ref{zero}, this equals $L_{T_{r-1}}\cdot (T_r-P_{r-1},n]\cdot S^{P_r-1}$, which equals 0 by Theorem \ref{zero}[a], since $T_{r-1}\le T_r-P_{r-1}-1$.

The second type is very similar: $\s(r)=r$ and $\s(r+1)=r-1$. If $\s(r-1)=i$ with $i<r-1$, then the summand of  (\ref{permsum}) can be written as $L_{T_{r-2}}\cdot (T_{r-1}-P_i,n]\cdot S^{P_{r-1}-1}$, which equals 0 by Theorem \ref{zero}[a].
Here we have used that $x_{T_{r-1}}^{P_i}x_{T_r}^{P_r}=(T_{r-1}-P_i,n-1]$.

The third type has $\s(r-1)=r-1$ and $\s(r+1)=r$. If $\s(r)=i\ne0$, then we have $L_{T_{r-1}}\cdot (T_r-P_i,n]\cdot S^{P_r-1}$, so again we get 0 by part [a]. If $\s(r)=0$, then $\s(i)=i$ for $1\le i\le r-1$, and our expression is of the form $[1,T_{r-1}]\cdot[n-1,n]\cdot S^{P_r-1}$ since $T_r=n-1$. This equals 0 by part [b] of Theorem \ref{zero}.

The fourth type has $\s(r-1)=r-1$ and $\s(r)=r$. Let $\s(i)=i$ for $k\le i\le r$ and, if $k>1$, $\s(k-1)<k-1$. Then $\s(r+1)=k-1$. If $\s(k-1)=j>0$, then our expression equals $L_{T_{k-2}}\cdot(T_{k-1}-P_j,n]\cdot S^{P_{k-1}-1}$, which equals 0  by \ref{zero}[a]. If $\s(k-1)=0$, then the expression equals $[1,T_{k-2}]\cdot [T_{k-1},n]\cdot S^{P_{k-1}-1}$, which is 0 by part [b]. If $k=1$, then $\s(r+1)=0$ and the expression equals $x_1\cdots x_n$, which is the only nonzero summand of (\ref{permsum}).

\section{Other manifolds} \label{mfsec}
In this section, we discuss our attempts to find orientable $n$-manifolds with $n\equiv0$ mod 4 and not a 2-power which have $\wbar_{n-\alh(n)}\ne0$.

First we discuss the real Bott manifold $B_{12}$ with relations as in Theorem \ref{main}. One way of showing $\wbar_9(B_{12})=0$ is by the method of Section \ref{pfsec}. We verified that $\chi\sq^9(x_ix_{11}x_{12})=0$ for $1\le i\le10$, which are the only classes needing consideration.

Another way is computer dependent. By \cite{Ds},
$$w_i(B_{12})=\s_i(x_1,\ldots,x_{10},x_1+\cdots+x_{10}),
$$
where $\s_i$ denotes the elementary symmetric polynomial. We use {\tt Maple} to compute this for $i=2$, 3, 5, and 9, and then to reduce each to a sum of $i$-fold products of distinct $x_j$'s using the relations $x_j^2=x_{j-1}x_j$. Then we use that $\wbar_9=w_9+w_2^2w_5+w_3^3$ for orientable manifolds, and use the relations again to obtain $\wbar_9=0$.

We tried several other 12-by-12 matrices $A$, such as the one whose only nonzero entries are $a_{i,i+1}=a_{i,i+2}=1$ for $1\le i\le 10$, and always obtained $\wbar_9=0$.

We also looked at generalized Dold manifolds, but the values of $n$ for which we found examples of the desired type formed a small subset of those in Theorem \ref{main} and Corollary \ref{cormain}.

\begin{defin} The generalized Dold manifold $P(n;m_1,\ldots,m_r)$ is defined to be $$S^{n}\times_T(CP^{m_1}\times\cdots\times CP^{m_r}),$$ where the involution $T$ is antipodal on $S^n$ and conjugation on the complex projective spaces.\end{defin}
Dold manifolds have been studied by many people, but our source for generalized Dold manifolds is \cite{SZ}, where the following result is proved.
\begin{thm} \label{SZthm} $($\cite[Corollary 5.4]{SZ}$)$ Let $c$ and $d_i$ satisfy $|c|=1$ and $|d_i|=2$. Then
$$H^*(P(n;m_1,\ldots,m_r);\zt)=\zt[c,d_1,\ldots,d_r]/(c^{n+1},d_1^{m_1+1},\ldots,d_r^{m_r+1}).$$
The total Stiefel-Whitney class is $$w(P(n;m_1,\ldots,m_r))
=(1+c)^{n+1-r}\prod(1+c+d_i)^{m_i+1}.$$\end{thm}

Our first example, an ordinary Dold manifold, was already noted in \cite{MP}.
We omit our proof.
\begin{thm}\label{2^e+1} For $n=2^e+1$, $P(2^e-3;2)$ is an orientable $n$-manifold with $\wbar_{n-\alh(n)}\ne0$.\end{thm}

\begin{thm}\label{two} If $n=3+2^e+2^{f+1}$ or $3+2^e+2^{f+1}+2^{g+1}$ with $2\le e\le f<g$, then 
$P(2^e-1;2,2^f)$ and $P(2^e-1;2,2^f,2^g)$, respectively, are orientable $n$-manifolds with $\wbar_{n-\alh(n)}\ne0$.
\end{thm}
\begin{proof} We sketch the proof in the first case. The second case is proved similarly.

We have 
$$w=(1+c)^{2^e-2}(1+c+d_1)^3(1+c+d_2)^{2^f+1}.$$
The manifold is orientable. We will show that the coefficient of $c^{2^e-2}d_1^1d_2^{2^f-1}$, a class of grading $2^e-2+2^{f+1}=n-\alh(n)$ in $\wbar$, is odd.

The dual class $\wbar$ can be written as
$$(1+c)^{2^L-2^{e+1}+2^e+2}(1+c+d_1)^{2^L-4}(1+c+d_1)(1+c+d_2)^{2^L-2^{f+1}}(1+c+d_2)^{2^f-1},$$
for $L$ sufficiently large.
The $d_2^{2^f-1}$ must come from the last factor, and $d_1^1$ from the third. The $d_1$ and $d_2$ can now be omitted from the other factors, whose product becomes $(1+c)^p$ with $p=3\cdot2^L-2^{f+1}-2^{e+1}+2^e-2$, and $\binom p{2^e-2}$ is odd.

\end{proof}

This procedure will not work for $n=3+2^e+2^{f+1}+2^{g+1}+2^{h+1}$ because the factor $(1+c)^{2^e-4}$ in $w$ leads to a $(1+c)^{2^e+4}$ in $\wbar$, which when combined with $(1+c+d_1)^{2^L-4}$  yields $(1+c)^{2^L+2^e}$, which equals 1 on $RP^{2^e-1}$. A {\tt Maple} check has led us to feel that there are no other values of $n$ for which there is an $n$-dimensional orientable generalized Dold manifold with $\wbar_{n-\alh(n)}\ne0$.

In \cite{B1} and \cite{B2}, dual Stiefel-Whitney classes of some non-orientable manifolds are obtained.

\def\line{\rule{.6in}{.6pt}}

\end{document}